\documentclass[final,oneeqnum,onethmnum,onefignum,onetabnum,dvips]{siamltex}
\usepackage{amsmath,amsfonts,amssymb,enumerate,graphicx,ifthen,bm,needspace}
\usepackage[hmarginratio=1:1,vmarginratio=1:1]{geometry}

\newcommand{\smallfigwid}{200pt}
\newcommand{\mediumfigwid}{300pt}
\newcommand{\largefigwid}{400pt}

\newcommand{\reaction}[1]{\overset{#1}{\rightarrow}}
\newcommand{\revreaction}[2]{\overset{#1}{\underset{#2}{\rightleftharpoons}}}
\newcommand{\dotprod}{\bullet}
\newcommand{\normal}{\bm{\nu}}

\newcommand{\mymbox}[1]{\quad\mbox{#1}\quad}
\newcommand{\fracslash}[2]{{#1}/{#2}}
\newcommand{\sfrac}[2]{\textstyle\frac{#1}{#2}\displaystyle}
\newcommand{\vep}{\varepsilon}
\newcommand{\st}{~:~}
\newcommand{\shalf}{\sfrac{1}{2}}

\newcommand{\rb}[1]{\left(#1\right)}
\newcommand{\sqb}[1]{\left[#1\right]}
\newcommand{\setb}[1]{\left\{#1\right\}}
\newcommand{\Cr}[2]{\ifthenelse{\equal{#2}{}}{C^{#1}}{C^{#1}\!\rb{#2}}}

\newcommand{\Rd}[1]{\mathbb{R}^{#1}}

\renewcommand{\vec}[1]{\textbf{#1}}
\newcommand{\mat}[1]{\textbf{#1}}
\renewcommand{\diag}[1]{\textsf{diag}\rb{#1}}

\newcommand{\integral}[4]{\int_{#1}^{#2} \, #3 \, d#4}
\newcommand{\deriv}[2]{\frac{d#1}{d#2}}
\newcommand{\derivslash}[2]{\fracslash{d#1}{d#2}}

\newcommand{\pderiv}[2]{\frac{\partial#1}{\partial#2}}
\newcommand{\pderivslash}[2]{\fracslash{\partial#1}{\partial#2}}

\newcommand{\seq}[4]{\setb{{#1}_{#2}}_{\ifthenelse{\equal{#3}{}}{}{#2=#3}}^{#4}}

\newcommand{\bigO}[1]{\cO\!\rb{#1}}
\newcommand{\litO}[1]{\textsf{o}\!\rb{#1}}
\renewcommand{\lim}{\operatornamewithlimits{\textsf{lim}}}
\renewcommand{\limsup}{\operatornamewithlimits{\textsf{lim\,sup}}}
\renewcommand{\liminf}{\operatornamewithlimits{\textsf{lim\,inf}}}

\newcommand{\sgn}[1]{\textsf{sgn}\!\rb{#1}}

\newcommand{\lnb}[1]{\textsf{ln}\!\rb{#1}}
\renewcommand{\exp}[1]{\textsf{e}^{#1}}
\newcommand{\expb}[1]{\textsf{exp}\!\rb{#1}}

\newcommand{\LambertW}[1]{W\!\rb{#1}}

\newcommand{\cC}{\mathcal{C}}
\newcommand{\cM}{\mathcal{M}}

\newcommand{\cO}{\mathcal{O}}
\newcommand{\cY}{\mathcal{Y}}

\newcommand{\bx}{\overline{x}}
\newcommand{\by}{\overline{y}}

\newcommand{\hrho}{\widehat{\rho}}

\newcommand{\twobyonematrix}[2]{\begin{pmatrix}#1\\#2\end{pmatrix}}
\newcommand{\twobytwomatrix}[4]{\begin{pmatrix}#1&#2\\#3&#4\end{pmatrix}}

\newtheorem{rem}[theorem]{\it{Remark}}
\newenvironment{remark}[1]{\begin{rem}\textnormal{#1}\end{rem}}

\newtheorem{rems}[theorem]{\it{Remarks}}
\newenvironment{remarks}[1]{\begin{rems}\textnormal{#1}\end{rems}}

\newtheorem{claim}[theorem]{\sc{Claim}}

\renewenvironment{proof}{\medskip\noindent\em Proof:
\rm}{\hspace*{\fill}$\square$\medskip}

\title{Properties of the Lindemann Mechanism in Phase Space}

\author{Matt S. Calder\footnotemark[2]~\footnotemark[3] \and David Siegel\footnotemark[2]~\footnotemark[4]}

\begin{document}

\maketitle

\renewcommand{\thefootnote}{\fnsymbol{footnote}}

\footnotetext[2]{Department of Applied Mathematics, University of
Waterloo, 200 University Avenue West, Waterloo, Ontario N2L 3G1,
Canada}

\footnotetext[3]{Research partially supported by an Ontario Graduate
Scholarship.}

\footnotetext[4]{Research supported by a Natural Sciences and
Engineering Research Council of Canada Discovery Grant.}

\renewcommand{\thefootnote}{\arabic{footnote}}

\begin{abstract}
We study the planar and scalar reductions of the nonlinear Lindemann
mechanism of unimolecular decay. First, we establish that the
origin, a degenerate critical point, is globally asymptotically
stable. Second, we prove there is a unique scalar solution (the slow
manifold) between the horizontal and vertical isoclines. Third, we
determine the concavity of all scalar solutions in the nonnegative
quadrant. Fourth, we establish that each scalar solution is a centre
manifold at the origin given by a Taylor series. Moreover, we
develop the leading-order behaviour of all planar solutions as time
tends to infinity. Finally, we determine the asymptotic behaviour of
the slow manifold at infinity by showing that it is a unique centre
manifold for a fixed point at infinity.
\end{abstract}



\pagestyle{myheadings} \thispagestyle{plain} \markboth{Matt S.
Calder and David Siegel}{Properties of the Lindemann Mechanism in
Phase Space}

\section{Introduction} \label{sec001}

A unimolecular reaction occurs when a single molecule undergoes a
chemical change. For unimolecular decay (or isomerization) to occur,
a certain amount of energy must be supplied externally, namely the
activation energy. For some time, there was debate concerning just
how the molecules became activated. Frederick Lindemann suggested
\cite{Lindemann} in 1922 that unimolecular decay involves two steps,
namely the activation/deactivation by collision step and the
reaction step. Cyril Norman Hinshelwood made further contributions
\cite{Hinshelwood} to the Lindemann model in 1926 and, consequently,
the Lindemann mechanism is occasionally referred to as the
Lindemann-Hinshelwood mechanism. For general references on
unimolecular reactions and the Lindemann mechanism, see, for
example, \cite{Benson,Forst,GilbertSmith,MoorePearson}.

Suppose that the reactant $A$ is to decay into the product $P$.
Then, according to the (nonlinear, self-activation) Lindemann
mechanism, $A$ is activated by a collision with itself producing the
activated complex $B$. This activation can also be reversed. The
complex then decays into the product. Symbolically,
\begin{equation} \label{eq10.001}
    A + A \revreaction{k_1}{k_{-\!1}} A + B,
    \quad
    B \reaction{k_2} P,
\end{equation}
where $k_1$, $k_{-\!1}$, and $k_2$ are the rate constants.

\subsection{Differential Equations and Common Approximations}

Using the Law of Mass Action, the concentrations of $A$ and $B$ in
\eqref{eq10.001} satisfy the planar reduction
\begin{equation} \label{eq10.004}
    \deriv{a}{\tau} = k_{-\!1} a b - k_1 a^2,
    \quad
    \deriv{b}{\tau} = k_1 a^2 - k_{-\!1} a b - k_2 b,
\end{equation}
where $\tau$ is time. The traditional initial conditions are
${a(0)=a_0}$ and ${b(0)=0}$. However, we will allow the initial
condition for $b$ to be arbitrary. Note that
${\derivslash{p}{\tau}=k_2b}$ and (traditionally) ${p(0)=0}$. Since
the differential equations \eqref{eq10.004} do not depend on the
differential equation for $p$, we need only consider the
differential equations for $a$ and $b$.

There are two common approximations for the planar reduction. The
Equilibrium Approximation (EA) and the Quasi-Steady-State
Approximation (QSSA), which have proved successful for the
Michaelis-Menten mechanism of an enzyme-substrate reaction
\cite{MichaelisMenten}, have also been applied to the Lindemann
mechanism. See, for example, \S2.2 of \cite{GilbertSmith} and pages
122--126 and 313--317 of \cite{MoorePearson}. These approximations
are frequently employed to simplify more complicated networks in
chemical kinetics which may involve, for example, inhibition or
cooperativity effects. For the EA, one assumes
${\derivslash{a}{\tau} \approx 0}$ for sufficiently large time. This
yields
\[
    b(\tau) \approx \frac{k_1}{k_{-\!1}} \, a(\tau).
\]
The QSSA, on the other hand, assumes ${\derivslash{b}{\tau} \approx
0}$ for sufficiently large time. This yields
\[
    b(\tau) \approx \frac{ k_1 a(\tau)^2 }{ k_2 + k_{-\!1} a(\tau) }.
\]

It will be useful for us to convert the planar reduction to
dimensionless form. Define
\[
    t := k_2 \tau,
    \quad
    x := \rb{ \frac{k_1}{k_2} } a,
    \quad
    y := \rb{ \frac{k_1}{k_2} } b,
    \mymbox{and}
    \vep := \frac{k_{-\!1}}{k_1} > 0,
\]
which are all dimensionless. Thus, $t$, $x$, and $y$ are,
respectively, a scaled time, reactant concentration, and complex
concentration. Moreover, the parameter ${\vep>0}$ measures how slow
the deactivation of the reactant is compared to the activation.
Traditionally, one may want to consider $\vep$ to be small. In our
analysis, the size of $\vep$ does not matter.

It is easy to verify that, with the above rescaling, the planar
reduction \eqref{eq10.004} becomes
\begin{equation} \label{eq11.001}
    \dot{x} = -x^2 + \vep x y,
    \quad
    \dot{y} = x^2 - \rb{ 1 + \vep x } y,
\end{equation}
where $\dot{}=\derivslash{}{t}$. Observe that the system
\eqref{eq11.001} is a regular perturbation problem. Occasionally, we
will need to refer to the vector field of this planar system. Hence,
define
\begin{equation} \label{eq10.009}
    \vec{g}(\vec{x}) :=
    \twobyonematrix{ -x^2 + \vep x y }{ x^2 - \rb{ 1 + \vep x } y },
\end{equation}
where ${\vec{x} := (x,y)^T}$. Moreover, we will be working with the
scalar reduction
\begin{equation} \label{eq11.002}
    y' = \frac{ x^2 - \rb{ 1 + \vep x } y }{ -x^2 + \vep x y },
\end{equation}
where $'=\derivslash{}{x}$, which describes solutions of the planar
reduction \eqref{eq11.001} in the $xy$-plane by suppressing the
dependence on time. We will need to refer to the right-hand side of
the scalar reduction. Hence, define
\begin{equation} \label{eq10.011}
    f(x,y) := \frac{ x^2 - \rb{ 1 + \vep x } y }{ -x^2 + \vep x y }.
\end{equation}

\begin{remark}
The function $f(x,y)$ can be written
${f(x,y)=\fracslash{g_2(x,y)}{g_1(x,y)}}$, where $g_1(x,y)$ and
$g_2(x,y)$ are the components of the function $\vec{g}(\vec{x})$
given in \eqref{eq10.009}.  Note the use of the row vector $(x,y)$
in the arguments of $g_1$ and $g_2$ as opposed to the column vector
$\vec{x}$.  To alleviate notational headaches that arise from
competing conventions involving row and column vectors, when there
will be no confusion we will use the notation appropriate for the
given situation.
\end{remark}

\subsection{Discussion}

The Lindemann mechanism has been explored mathematically by others.
For example, the planar system \eqref{eq11.001} has been treated as
a perturbation problem in
\cite{RichardsonVolkLauLinEyring,ShinGiddings}. Furthermore, Simon
Fraser has used the Lindemann mechanism \cite{Fraser1988,Fraser2004}
as an example in his work on the dynamical systems approach to
chemical kinetics. Finally, properties of the Lindemann mechanism
have been explored mathematically in \cite{delaSelvaPina}.

The focus of this paper is the detailed behaviour of solutions to
the planar reduction \eqref{eq11.001} in phase space. That is, we
perform a careful phase-plane analysis to reveal important details
that a common phase-plane analysis would miss. Equivalently, we are
studying solutions of the scalar reduction \eqref{eq11.002}. It is
worth reiterating that our analysis does not depend on the size of
the parameter $\vep$ (which is traditionally treated as being
small). In \S\ref{sec002}, we present the basic phase portrait in
the nonnegative quadrant. Moreover, we establish that the origin is
a saddle node and is globally asymptotically stable with respect to
the nonnegative quadrant. In \S\ref{sec003}, we describe the
isocline structure which we exploit in later sections. For example,
the isocline structure plays an important role in determining the
concavity and asymptotic behaviour of solutions. In \S\ref{sec004},
we prove that there is a unique slow manifold $\cM$ between the
horizontal and vertical isoclines. To this end, we use a nonstandard
version of the Antifunnel Theorem. In \S\ref{sec005}, we determine
the concavity of all solutions, excluding the slow manifold, in the
nonnegative quadrant by analyzing an auxiliary function. In
\S\ref{sec006}, we use the Centre Manifold Theorem to show that all
scalar solutions are given by a Taylor series at the origin.
Moreover, we establish the leading-order behaviour of planar
solutions as ${t\to\infty}$. This is nontrivial due to the fact that
the origin is a degenerate critical point. In \S\ref{sec007}, we
show that all planar solutions must enter and remain in the region
bounded by the horizontal isocline and the isocline for the slope of
the slow manifold at infinity. In \S\ref{sec008}, we single out
properties of the slow manifold. These properties include concavity,
monotonicity, and asymptotic behaviour at the origin and at
infinity. Finally, in \S\ref{sec009}, we state some open problems.

\section{Phase Portrait} \label{sec002}

A computer-generated phase portrait for the planar reduction
\eqref{eq11.001}, restricted to the physically relevant and
positively invariant nonnegative quadrant $S$, is given in
Figure~\ref{fig10.002}. In this paper, we will develop precise
mathematical properties of the phase portrait. Equivalently, we
develop results on solutions of the scalar reduction
\eqref{eq11.002}.

\begin{figure}[t]
\begin{center}
    \includegraphics[width=\mediumfigwid]{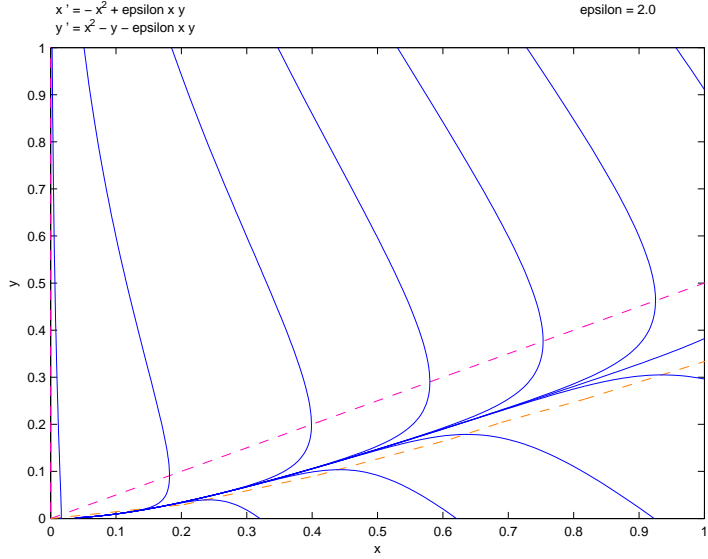}
    \caption{A phase portrait for \eqref{eq11.001} for parameter value ${\vep=2.0}$ along
    with the isoclines.} \label{fig10.002}
\end{center}
\end{figure}

The horizontal and vertical isoclines for the planar system
\eqref{eq11.001}, which are found by respectively setting
${\dot{y}=0}$ and ${\dot{x}=0}$, are given by
\begin{equation} \label{eq11.004}
    y = H(x) := \frac{x^2}{ 1 + \vep x }
    \mymbox{and}
    y = V(x) := \frac{x}{\vep}.
\end{equation}
The QSSA corresponds to the horizontal isocline (the
quasi-steady-state manifold) and the EA corresponds to the vertical
isocline (the rapid equilibrium manifold). Observe that
${H(0)=0=V(0)}$, both $H$ and $V$ are strictly increasing, and
${V(x)>H(x)}$ for all ${x>0}$. It appears from the phase portrait
that the region between the isoclines,
\[
    \Gamma_0 := \setb{ ~ (x,y) \st x > 0, ~ H(x) \leq y \leq V(x) ~ },
\]
acts like a trapping region for (time-dependent) solutions of the
planar reduction. Moreover, the origin appears to be globally
asymptotically stable.

\begin{theorem} \label{thm0011}
Consider the planar system \eqref{eq11.001}.
\begin{enumerate}[(a)]
    \item
        The region $\Gamma_0$ is positively invariant.
    \item
        Let $\vec{x}(t)$ be the solution with
        initial condition ${\vec{x}(0)=\vec{x}_0}$, where ${\vec{x}_0\in\setb{\vec{x} \in S \st x > 0}}$.
        Then, there is a ${t^* \geq 0}$ such
        that ${\vec{x}(t)\in\Gamma_0}$ for all ${t \geq t^*}$.
    \item
        Let $\vec{x}(t)$ be the solution with
        ${\vec{x}(0)=\vec{x}_0}$, where ${\vec{x}_0 \in S}$.
        Then, ${\vec{x}(t)\to\vec{0}}$ as ${t\to\infty}$.
\end{enumerate}
\end{theorem}

\needspace{2.0cm}
\begin{proof}
\begin{enumerate}[(a)]
    \item
        It follows from the definition \eqref{eq10.009} of the vector
        field $\vec{g}$ that ${\vec{g}\dotprod\normal<0}$ along $V$
        and $H$, where $\normal$ is the outward unit normal vector.  Thus, solutions cannot exit $\Gamma_0$ through
        the horizontal or vertical isoclines. Furthermore, solutions cannot escape from $\Gamma_0$ through the origin since solutions
        do not intersect.  Hence, $\Gamma_0$ is positively invariant.
    \item
        We will break the proof into cases.

        \begin{description}
            \item[Case 1: ${(x_0,y_0)\in\Gamma_0}$.]
                Since $\Gamma_0$ is positively invariant, ${\vec{x}(t)\in\Gamma_0}$ for all ${t \geq
                0}$.
            \item[Case 2: ${x_0>0}$ and ${y_0>V(x_0)}$.]
                Suppose, on the contrary, that $\vec{x}(t)$ does not enter
                $\Gamma_0$. It follows that ${y(t)>V(x(t))}$ for
                all ${t \geq 0}$. Using the differential equation
                \eqref{eq11.001}, we know ${\dot{x}(t)>0}$ and
                ${\dot{y}(t)<0}$ for all ${t \geq 0}$. Now, we see
                from the definition \eqref{eq10.011} of the function $f$ that
                \[
                    \frac{\dot{y}(t)}{\dot{x}(t)}
                    = f(x(t),y(t))
                    = -1 - \frac{y(t)}{ \vep x(t) y(t) - x(t)^2 }
                    < -1
                    \mymbox{for all}
                    t \geq 0.
                \]
                Note that ${\vep x y - x^2>0}$ since ${y>V(x)=\fracslash{x}{\vep}}$ and ${x,y>0}$. Thus,
                \[
                    \dot{y}(s) < -\dot{x}(s)
                    \mymbox{for all}
                    s \geq 0.
                \]
                Integrating with respect to $s$ from $0$ to $t$ and
                rearranging, we obtain
                \[
                    y(t) \leq y_0 - \sqb{ x(t) - x_0 }
                    \mymbox{for all}
                    t \geq 0.
                \]

                Let $\rb{x_1,V(x_1)}$ be the point of intersection of the vertical
                isocline ${y=V(x)}$ and the straight line
                ${y=y_0-\rb{x-x_0}}$. Obviously,
                ${x_1>x_0}$. Since $x(t)$ is
                monotone increasing and bounded above by $x_1$, we see that there is
                an ${\bx\in\sqb{x_0,x_1}}$ such that ${x(t)\to\bx}$
                as ${t\to\infty}$. Similarly, since $y(t)$ is
                monotone decreasing and bounded below by $V(x_0)$, we see that there
                is a ${\by\in\sqb{V(x_0),y_0}}$ such that
                ${y(t)\to\by}$ as ${t\to\infty}$. Thus, the $\omega$-limit set is
                ${\omega(x_0,y_0)=\setb{(\bx,\by)}}$. Since $\omega(x_0,y_0)$ is
                invariant and $(0,0)$ is the only equilibrium point of the system,
                ${\bx=0}$ and ${\by=0}$. This is a contradiction.
            \item[Case 3: ${x_0>0}$ and ${0 \leq y_0 < H(x_0)}$.]
                This case is proved in a manner similar to Case~2.
        \end{description}
    \item
        If ${x_0=0}$, the solution of
        \eqref{eq11.001} is ${\vec{x}(t)=\rb{0,y_0\exp{-t}}^T}$.
        This clearly satisfies ${\vec{x}(t)\to\vec{0}}$ as
        ${t\to\infty}$. Thus, we can assume ${x_0>0}$ and, by virtue of part (b), we can assume
        further that ${(x_0,y_0)\in\Gamma_0}$. It follows from
        the differential equation \eqref{eq11.001} and
        the fact that $\Gamma_0$ is positively invariant that
        ${\dot{x}(t) \leq 0}$ and ${\dot{y}(t) \leq 0}$ for all ${t \geq
        0}$.  Since both $x(t)$ and $y(t)$ are decreasing
        and bounded below by zero, by the Monotone Convergence
        Theorem we know that there are $\bx$ and $\by$ such that
        ${x(t)\to\bx}$ and ${y(t)\to\by}$ as ${t\to\infty}$. Thus, the $\omega$-limit set is
        ${\omega(x_0,y_0)=\setb{(\bx,\by)}}$. Since $\omega(x_0,y_0)$ is
        invariant and $(0,0)$ is the only equilibrium point of the system,
        ${\bx=0}$ and ${\by=0}$.
\end{enumerate}
\end{proof}

The Jacobian matrix at the origin for the planar system
\eqref{eq11.001} is $\diag{0,-1}$. Thus, the origin is a
nonhyperbolic fixed point. The Hartman-Grobman Theorem,
unfortunately, cannot be applied here. Using Theorem~65 in \S9.21 of
\cite{Andronov}, the origin is a saddle node which consists of two
hyperbolic sectors and one parabolic sector. As we will effectively
show later, $S$ is contained in the parabolic sector.

\section{The Isocline Structure} \label{sec003}

The horizontal and vertical isoclines, along with all isoclines
between them, will be very useful. If we solve ${f(x,y)=c}$ for $y$,
we obtain ${y=F(x,c)}$, where
\begin{equation} \label{eq11.006}
    F(x,c) := \frac{ x^2 }{ K(c) + \vep x },
    \quad
    c \ne -1,
    \quad
    x \ne -\vep^{-1} K(c),
\end{equation}
and
\begin{equation} \label{eq11.005}
    K(c) := \frac{1}{1+c},
    \quad
    c \ne -1.
\end{equation}
That is, ${y=F(x,c)}$ is the isocline for slope $c$.
Figure~\ref{fig11.001} gives a sketch of $K$. Note that each
isocline, for ${c\in\Rd{}\backslash\setb{-1}}$, has a vertical
asymptote at ${x=-\vep^{-1}K(c)}$.

\begin{figure}[t]
\begin{center}
    \includegraphics[width=\smallfigwid]{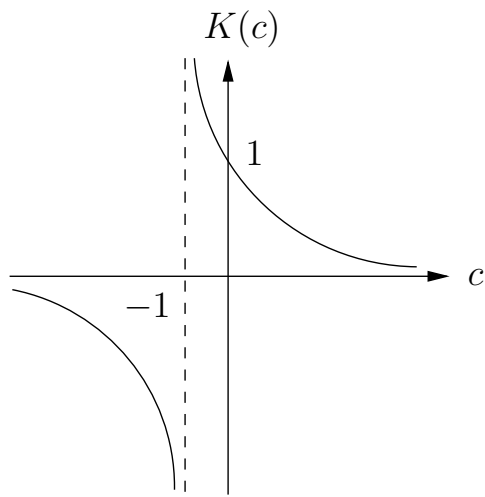}
    \caption{Graph of the function $K(c)$.} \label{fig11.001}
\end{center}
\end{figure}

\begin{samepage}
\begin{remarks}~
\begin{enumerate}[(i)]
    \item
        The interior of the region $\Gamma_0$ corresponds to ${0<c<\infty}$ and
        ${0<K(c)<1}$.
    \item
        Two exceptional isoclines are ${y=V(x)}$ (the vertical isocline) and ${y=0}$
        which correspond, respectively, to
        \[
            \lim_{c\to\infty} F(x,c) = \frac{x}{\vep}
            \mymbox{and}
            \lim_{c\to-1} F(x,c) = 0.
        \]
\end{enumerate}
\end{remarks}
\end{samepage}

\begin{claim}
Let ${c\in\Rd{}\backslash\setb{-1}}$ and let ${w(x):=F(x,c)}$ be the
isocline for slope $c$. Then, the derivative of $w$ satisfies
\begin{equation} \label{eq11.007}
    \lim_{x\to\infty} w'(x) = \vep^{-1}.
\end{equation}
Furthermore, $w$ is concave up for all ${x>-\vep^{-1}K(c)}$ and
satisfies the differential equation
\begin{equation} \label{eq11.008}
    x^2 w' + w \rb{ \vep w - 2 x } = 0.
\end{equation}
\end{claim}

\begin{proof}
The proof is straight-forward and omitted.
\end{proof}

\begin{remark}
The vertical isocline satisfies the limit \eqref{eq11.007} and the
differential equation \eqref{eq11.008}. The isocline ${w(x)=0}$ (the
isocline for slope $-1$) also satisfies the differential equation
but does not satisfy the limit.
\end{remark}

The isocline structure is sketched in Figure~\ref{fig11.002}. We
will often appeal to the isocline structure. For example, if a
scalar solution of \eqref{eq11.002} is above the line ${y=0}$ and
below the horizontal isocline ${y=H(x)}$, we know that
${-1<y'(x)<0}$.

\begin{figure}[t]
\begin{center}
    \includegraphics[width=\largefigwid]{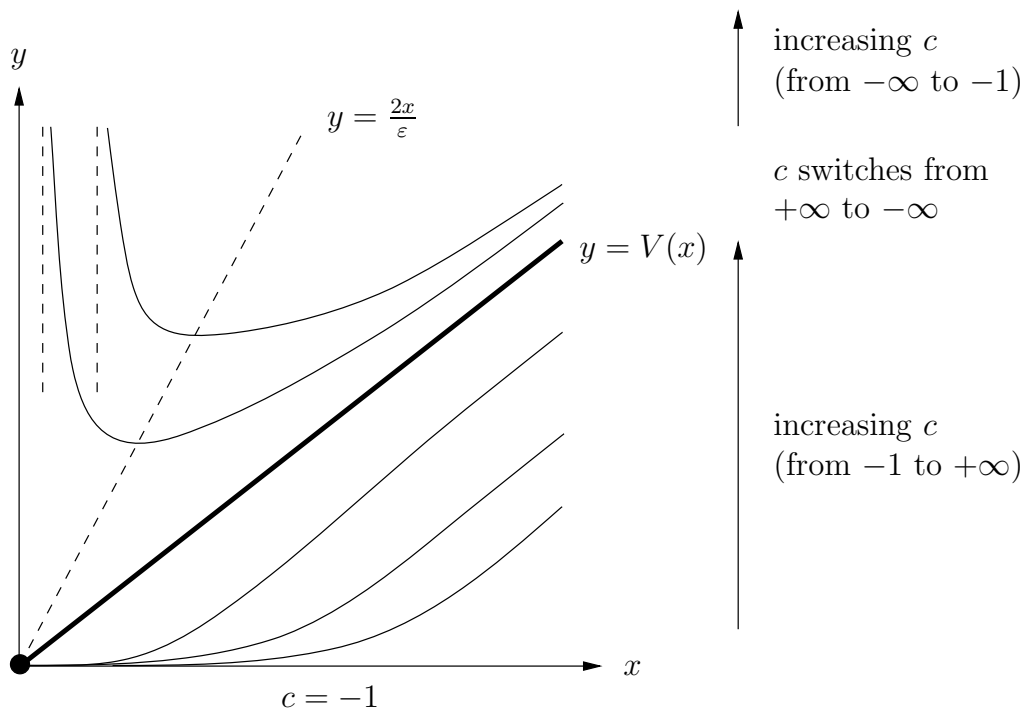}
    \caption{Sketch of the isocline structure of \eqref{eq11.002}. The
    isoclines above the vertical isocline have zero slope along the
    line ${y=\fracslash{2x}{\vep}}$.} \label{fig11.002}
\end{center}
\end{figure}

\section{Existence and Uniqueness of the Slow Manifold} \label{sec004}

It appears from the given phase portrait, Figure~\ref{fig10.002},
that there exists a unique solution to \eqref{eq11.002} that lies
entirely in the region between the horizontal and vertical
isoclines. To prove this, we will need to use a nonstandard version
of the Antifunnel Theorem. See, for example, Chapters~1 and 4 of
\cite{HubbardWest}.

\begin{definition}
Let ${I=[a,b)}$ or ${I=(a,b)}$ be an interval (where
${a<b\leq\infty}$) and consider the first-order differential
equation ${y'=f(x,y)}$ over $I$. Let
${\alpha,\beta\in\Cr{1}{I,\Rd{}}}$ be functions satisfying
\begin{equation} \label{eq-001}
    \alpha'(x) \leq f(x,\alpha(x))
    \mymbox{and}
    f(x,\beta(x)) \leq \beta'(x)
    \mymbox{for all}
    x \in I.
\end{equation}
\begin{enumerate}[(a)]
    \item
        The curves $\alpha$ and $\beta$ satisfying \eqref{eq-001} are, respectively,
        a \emph{lower fence} and an \emph{upper fence}. If
        there is always a strict inequality in \eqref{eq-001}, the fences are
        \emph{strong}.  Otherwise, the fences are \emph{weak}.
    \item
        If ${\beta(x)<\alpha(x)}$ on $I$, then the set $\Gamma$ is called an
        \emph{antifunnel}, where
        \[
            \Gamma := \setb{ ~ (x,y) \st x \in I, ~ \beta(x) \leq y \leq \alpha(x) ~ }.
        \]
\end{enumerate}
\end{definition}

\begin{theorem}[Antifunnel Theorem, p.196 of \cite{HubbardWest}] \label{thmD.003}
Let $\Gamma$ be an antifunnel with strong lower and upper fences
$\alpha$ and $\beta$, respectively, for the differential equation
${y'=f(x,y)}$ over the interval $I$, where ${I=[a,\infty)}$ or
${I=(a,\infty)}$. Suppose that there is a function $r$ such that
\[
    r(x) < \pderiv{f}{y}(x,y) \mymbox{for all} (x,y) \in \Gamma
    \mymbox{and}
    \lim_{x\to\infty} \frac{ \alpha(x) - \beta(x) }{ \expb{ \integral{a}{x}{r(s)}{s} } } = 0.
\]
Then, there exists a unique solution $y(x)$ to the differential
equation ${y'=f(x,y)}$ which satisfies ${\beta(x)<y(x)<\alpha(x)}$
for all ${x \in I}$.
\end{theorem}

\begin{remark}
The standard version of the Antifunnel Theorem applies to
antifunnels that are narrowing. That is, where ${\alpha(x)-\beta(x)
\to 0}$ as ${x\to\infty}$. This version applies to, for example, the
Michaelis-Menten mechanism \cite{CalderSiegel01}.
\end{remark}

\subsection{Existence-Uniqueness Theorem}

We want to show that there is a unique scalar solution that lies
entirely in the region $\Gamma_0$. However, the vertical isocline is
not a strong lower fence since ``${f(x,V(x))=\infty}$.'' This turns
out to be a fortunate obstacle.

Suppose that we want the isocline
\[
    w(x) := \frac{ x^2 }{ r + \vep x },
    \quad
    0 < r < 1
\]
to be a strong lower fence for the differential equation
\eqref{eq11.002} for all ${x>0}$. Note that the condition on $r$
restricts the isocline to being between the horizontal and vertical
isoclines. Note also that
\[
    f(x,w(x)) = K^{-1}(r)
    \mymbox{for all}
    x > 0.
\]
Since $w$ is concave up and satisfies the limit \eqref{eq11.007}, we
know that ${w'(x)<\vep^{-1}}$ for all ${x>0}$. Hence,
\[
    \vep^{-1} \leq K^{-1}(r)
    \implies
    w'(x) < f(x,w(x))
    \mymbox{for all}
    x > 0.
\]
Since we want the isocline that will give us the thinnest
antifunnel, we choose
\begin{equation} \label{eq11.009}
    \alpha(x) := \frac{ x^2 }{ K\rb{\vep^{-1}} + \vep x }.
\end{equation}
Note that $\alpha(x)$ is the isocline for slope $\vep^{-1}$ and
\[
    K\rb{\vep^{-1}} = \frac{\vep}{1+\vep}.
\]
Hence, define the region
\[
    \Gamma_1 := \setb{ ~ (x,y) \st x > 0, ~ H(x) \leq y \leq \alpha(x) ~ }.
\]

\begin{samepage}
\begin{theorem}~ \label{c10t-001}
\begin{enumerate}[(a)]
    \item
        There exists a unique solution ${y=\cM(x)}$ (the slow manifold) in
        $\Gamma_1$ for the scalar differential equation
        \eqref{eq11.002}.
    \item
        The solution ${y=\cM(x)}$ is also the only solution
        that lies entirely in $\Gamma_0$.
\end{enumerate}
\end{theorem}
\end{samepage}

\needspace{2.0cm}
\begin{proof}
\begin{enumerate}[(a)]
    \item
        We have already established that $\alpha$ is a strong lower fence. To show
        that $H$ is a strong upper fence, observe
        \[
            f(x,H(x)) = 0 < H'(x)
            \mymbox{for all}
            x > 0.
        \]
        Moreover, ${\alpha(x)>H(x)}$ for all ${x>0}$. By
        definition, $\Gamma_1$ is an antifunnel.

        Suppose that ${x>0}$ and ${H(x) \leq y < V(x)}$. Using
        \eqref{eq11.004},
        \[
            \frac{ \vep x^2 }{ 1 + \vep x } \leq \vep y < x
            \mymbox{and}
             0 < x - \vep y \leq x - \frac{ \vep x^2 }{ 1 + \vep x }.
        \]
        Rearranging, we have
        \[
            0 < \frac{ 1 + \vep x }{ x } \leq \frac{ 1 }{ x - \vep y }
            \mymbox{and}
            0 < \rb{ \frac{1}{x} + \vep }^2 \leq \frac{ 1 }{ \rb{ x - \vep y }^2 }.
        \]
        Thus, using the definition \eqref{eq10.011} of $f$ we have
        \[
            \pderiv{f}{y}(x,y)
            = \frac{1}{ \rb{ x - \vep y }^2 }
            \geq \rb{ \frac{1}{x} + \vep }^2
            > \vep^2.
        \]
        Hence, we can apply the Antifunnel Theorem with
        ${r(x):=\vep^2}$.  To see why, consider that
        \[
            \integral{0}{\infty}{r(x)}{x} = \infty
            \mymbox{and}
            \lim_{x \to \infty} \sqb{ \alpha(x) - H(x) } = \frac{1}{ \vep^2 \rb{ 1 + \vep } }
        \]
        and thus
        \[
            \lim_{x\to\infty} \frac{ \alpha(x) - H(x) }{ \expb{ \integral{0}{x}{r(s)}{s} } } = 0.
        \]
        We can therefore conclude that there is a unique solution ${y=\cM(x)}$
        to \eqref{eq11.002} that lies in $\Gamma_1$ for all ${x>0}$.
    \item
        Let $y$ be a solution in $\Gamma_0$ lying below $\cM$. Since
        $\cM$ is the only solution contained in $\Gamma_1$, $y$ must
        leave $\Gamma_1$ through the horizontal isocline.

        Now, let $y$ be a solution in $\Gamma_0$ lying above
        $\cM$. Suppose on the contrary that $y$ never leaves $\Gamma_0$
        and thus ${\cM(x)<y(x)<V(x)}$ for all ${x>0}$.
        Consider the isocline ${w(x):=F\rb{x,2\vep^{-1}}}$, which satisfies ${\alpha(x)<w(x)<V(x)}$. Since
        $w$ is a strong lower fence, the proof of the first part of
        the theorem can be adapted to show that $\cM$ is the only
        solution contained in the region between the isoclines $H$
        and $w$. Thus, there is an ${a>0}$ such that ${y(a)=w(a)}$. Now,
        ${w'(x)<\vep^{-1}}$ for all ${x>0}$ and ${y'(x) \geq
        2\vep^{-1}}$ if ${y(x) \geq w(x)}$. It follows from a simple comparison argument that
        ${w(x)<y(x)<V(x)}$ for all ${x>a}$. So,
        ${y(x)>y(a)+2\vep^{-1}\rb{x-a}}$ for all ${x>a}$. This is a
        contradiction since ${y(x)>V(x)}$ for sufficiently large
        $x$.
\end{enumerate}
\end{proof}

\needspace{2.0cm}
\begin{remarks}~
\begin{enumerate}[(i)]
    \item
        We are referring to the unique solution between the
        horizontal and vertical isoclines as \emph{the} slow manifold. However, all scalar
        solutions in $\Gamma_0$ are technically slow manifolds (and,
        as it turns out, centre manifolds).  This is because, as
        functions of time, the solutions approach the origin in the
        slow direction.
    \item
        There is no isocline $w(x)$ such that
        ${w(x)>H(x)}$ and $w(x)$ is a strong upper fence for all
        ${x>0}$. To see why this is the case, suppose ${w(x):=F(x,c)}$, where
        ${c>0}$, satisfies ${w'(x)>f(x,w(x))}$ for all ${x>0}$. This is
        impossible, since ${f(x,w(x))=c}$ for all ${x>0}$ and
        ${w'(x) \to 0}$ as ${x \to 0^+}$.
\end{enumerate}
\end{remarks}

\begin{proposition} \label{c10t-002}
Let $y$ be a solution to \eqref{eq11.002} lying inside $\Gamma_1$
for ${x\in(0,a)}$, where ${a>0}$. Then, we can extend $y(x)$ and
$y'(x)$ to say ${y(0)=0}$ and ${y'(0)=0}$.
\end{proposition}

\begin{proof}
Observe that
\[
    \lim_{x \to 0^+} \alpha(x) = 0
    \mymbox{and}
    \lim_{x \to 0^+} \frac{\alpha(x)}{x} = 0.
\]
Since
\[
    0 < y(x) < \alpha(x)
    \mymbox{for all}
    x \in \rb{0,a},
\]
the Squeeze Theorem establishes ${y(0)=0}$. Now,
\[
    0 < \frac{y(x)-y(0)}{x-0} < \frac{\alpha(x)}{x}
    \mymbox{for all}
    x \in \rb{0,a}.
\]
Thus, by the Squeeze Theorem again as well as the definition of
(right) derivative, we can say ${y'(0)=0}$.
\end{proof}

\subsection{Nested Antifunnels}

The region $\Gamma_1$ is the thinnest antifunnel for ${x>0}$ with
isoclines as boundaries. However, we can find thinner antifunnels
than $\Gamma_1$ which are valid for different intervals. For an
isocline ${w(x):=F(x,c)}$, where ${0<c<\vep^{-1}}$, to be a strong
lower fence on an interval, we need ${w'(x)<f(x,w(x))}$. Solving the
equation ${w'(x)=f(x,w(x))}$, as we shall see, gives ${x=\xi(c)}$,
where
\begin{equation} \label{eq11.010}
    \xi(c) := \sqb{ \frac{ K(c) }{ \vep } } \sqb{ \frac{ 1 }{ \sqrt{ 1 - \vep c } }  - 1 },
    \quad
    c \in \rb{0,\vep^{-1}}.
\end{equation}
Note that ${1 - \vep c > 0}$. We will quickly establish a few
properties of $\xi(c)$. See Figure~\ref{fig11.003} for a sketch of
the function.

\begin{figure}[t]
\begin{center}
    \includegraphics[width=\smallfigwid]{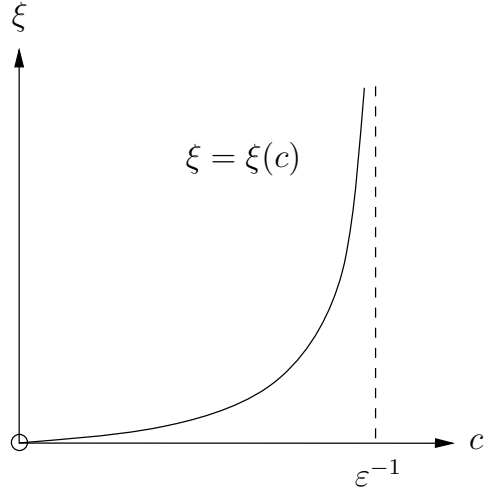}
    \caption[Graph of $\xi(c)$]{Graph of the function $\xi(c)$ for arbitrary ${\vep>0}$.} \label{fig11.003}
\end{center}
\end{figure}

\begin{samepage}
\begin{claim}~
\begin{enumerate}[(a)]
    \item
        The function $\xi(c)$ satisfies
        \begin{equation} \label{eq11.011}
            \lim_{c \to 0^+} \xi(c) = 0,
            \quad
            \lim_{ c \to \rb{\vep^{-1}}^- } \xi(c) = \infty,
            \mymbox{and}
            \xi'(c) > 0 \mymbox{for all} c \in \rb{0,\vep^{-1}}.
        \end{equation}
    \item
        The function $\xi(c)$ is analytic for all ${c\in\rb{0,\vep^{-1}}}$.
        Furthermore, $\xi(c)$ has analytic inverse $\xi^{-1}(x)$ defined for
        all ${x>0}$.
\end{enumerate}
\end{claim}
\end{samepage}

\begin{proof}
The proof is routine, tedious, and omitted.
\end{proof}

\begin{samepage}
\begin{proposition} \label{prop0016}
Let ${c\in\rb{0,\vep^{-1}}}$ and ${w(x):=F(x,c)}$.
\begin{enumerate}[(a)]
    \item
        The isocline $w$ satisfies
        \[
            w'(x)
            \begin{cases}
                < f(x,w(x)), \mymbox{if} 0 < x < \xi(c) \\
                = f(x,w(x)), \mymbox{if} x = \xi(c) \\
                > f(x,w(x)), \mymbox{if} x > \xi(c)
            \end{cases}.
        \]
    \item
        The slow manifold satisfies
        \[
            w(x) < \cM(x) < \alpha(x)
            \mymbox{for all}
            x > \xi(c).
        \]
\end{enumerate}
\end{proposition}
\end{samepage}

\needspace{2.0cm}
\begin{proof}
\begin{enumerate}[(a)]
    \item
        Note that ${f(x,w(x))=c}$ for all ${x>0}$. If we set
        ${w'(x)=c}$, we obtain
        \[
            \vep \rb{ 1 - \vep c } x^2 + 2 K(c) \rb{ 1 - \vep c } x - c K(c)^2 = 0.
        \]
        This has two roots, one negative and one positive. The
        positive root is given by ${x=\xi(c)}$ with $\xi(c)$ as in
        \eqref{eq11.010}. It is a routine matter to confirm that
        ${w'(x)<f(x,w(x))}$ when ${0<x<\xi(c)}$ and that
        ${w'(x)>f(x,w(x))}$ when ${x>\xi(c)}$.
    \item
        It follows from the Antifunnel Theorem.
\end{enumerate}
\end{proof}

\section{Concavity} \label{sec005}

In this section, we will establish the concavity of all scalar
solutions, except for the slow manifold, in the nonnegative
quadrant. The concavity of the slow manifold will be established
later. These results will be obtained by using an auxiliary
function. Moreover, we will construct a curve of inflection points
which approximates the slow manifold.

\subsection{Establishing Concavity}

Let $y$ be a solution to \eqref{eq11.002} and consider the function
$f$ given in \eqref{eq10.011}. If we differentiate
${y'(x)=f(x,y(x))}$ and apply the Chain Rule, we obtain
\begin{equation} \label{eq11.012}
    y''(x) = p(x,y(x)) h(x,y(x)),
\end{equation}
where
\begin{equation} \label{eq11.013}
    p(x,y) := \frac{1}{ x^2 \rb{ \vep y - x }^{2} }
    \mymbox{and}
    h(x,y) := x^2 f(x,y) + y \rb{ \vep y - 2 x }.
\end{equation}

The function $p(x,y)$ is positive everywhere except along the
vertical isocline and for ${x=0}$, where it is undefined. We will be
considering the functions $h(x,y)$ and $p(x,y)$ along a given
solution $y(x)$ so we will abuse notation by writing
${h(x):=h(x,y(x))}$ and ${p(x):=p(x,y(x))}$.

For a given ${x>0}$ with ${y(x) \ne V(x)}$, it follows from
\eqref{eq11.012} and the fact that ${p(x)>0}$ that the sign of
$h(x)$ is the same as the sign of $y''(x)$. Furthermore, if we
differentiate $h(x)$ with respect to $x$ and apply \eqref{eq11.012},
we see that the function $h$ has derivative
\begin{equation} \label{eq11.014}
    h'(x) = x^2 p(x) h(x) + 2 y(x) \sqb{ \vep y'(x) - 1 }.
\end{equation}

\begin{remark}
The function $h$ cannot tell us anything about the concavity of
solutions at ${x=0}$, not even by taking a limit.
\end{remark}

\begin{claim} \label{claim0006}
Let $y$ be a solution to \eqref{eq11.002} and let ${x_0>0}$ with
${y(x_0) \ne V(x_0)}$. Consider the isocline through the point
$(x_0,y(x_0))$, which is given by ${w(x):=F(x,y'(x_0))}$. Then,
\[
    h(x_0) = x_0 \sqb{ y'(x_0) - w'(x_0) }.
\]
Furthermore,
\[
    y''(x_0) > 0 \Longleftrightarrow y'(x_0) > w'(x_0)
    \mymbox{and}
    y''(x_0) < 0 \Longleftrightarrow y'(x_0) < w'(x_0).
\]
\end{claim}

\begin{proof}
The first part follows from \eqref{eq11.008} and \eqref{eq11.013}.
The second part follows from the first.
\end{proof}

The concavity of all solutions in all regions of the nonnegative
quadrant can be deduced using the auxiliary function $h$ and the
following easy-to-verify lemma. Table~\ref{tab11.001} summarizes
what we will develop in this section.

\begin{lemma} \label{lem4.001}
Let $I$ be one of the intervals $[a,b]$, $(a,b)$, $[a,b)$, and
$(a,b]$. Suppose that ${\phi \in C(I)}$ is a function having at
least one zero in $I$.
\begin{enumerate}[(a)]
    \item
        If ${I=(a,b]}$ or ${I=[a,b]}$, then the function $\phi$ has
        a right-most zero in $I$.  Likewise, if ${I=[a,b)}$ or
        ${I=[a,b]}$, then the function $\phi$ has a left-most zero in $I$.
    \item
        If ${\phi\in\Cr{1}{I}}$ and ${\phi'(x)>0}$ for every zero
        of $\phi$ in $I$, then $\phi$ has exactly one zero in $I$.
\end{enumerate}
\end{lemma}

\begin{table}[t]
\begin{center}
\begin{tabular}{|c|c|}
    \hline
    Region & Concavity of Solutions \\ \hline\hline
    $ 0 \leq y \leq H $ & concave down \\ \hline
    $ H < y < \cM $ & concave up, then inflection point, then concave down \\ \hline
    $ \cM < y < V $ & concave up \\ \hline
    $ y > V $ & concave up, then inflection point, then concave down \\ \hline
\end{tabular}
\caption{A summary of the concavity of solutions of \eqref{eq11.002}
in the nonnegative quadrant.} \label{tab11.001}
\end{center}
\end{table}

\begin{proposition}
Let $y$ be a solution to \eqref{eq11.002} lying below $H$ with
domain $[a,b]$, where ${0<a<b}$, ${y(a)=H(a)}$, and ${y(b)=0}$.
Then, $y$ is concave down on $[a,b]$.
\end{proposition}

\begin{proof}
We know ${y'(a)=0}$ and ${y'(b)=-1}$. Also, ${y(x)>0}$ and
${y'(x)<0}$ for all ${x\in(a,b)}$. Note that ${y<V(x)}$ which
implies ${\vep y - 2 x < 0}$ for all ${x\in[a,b]}$. Let $h$ be as in
\eqref{eq11.013} defined with respect to the solution $y$. Observe
that
\[
    h(a) = H(a) \sqb{ \vep H(a) - 2a } < 0
    \mymbox{and}
    h(b) = -b^2 < 0.
\]
Observe also that
\[
    h(x) = x^2 y'(x) + y(x) \sqb{ \vep y(x) - 2 x } < 0
    \mymbox{for all}
    x \in (a,b).
\]
Therefore, $y$ is concave down on $[a,b]$.
\end{proof}

\begin{proposition}
Let $y$ be a solution to \eqref{eq11.002} lying above $H$ and below
$\cM$ with domain $(0,a]$, where ${a>0}$ and ${y(a)=H(a)}$. Then,
there is a unique ${x_1\in(0,a)}$ such that ${y''(x_1)=0}$.
Moreover, $y$ is concave up on $(0,x_1)$ and concave down on
$(x_1,a]$.
\end{proposition}

\begin{proof}
Let $h$ be as in \eqref{eq11.013} defined with respect to the
solution $y$.  Now, we know ${y'(a)=0}$ and, by
Proposition~\ref{c10t-002}, we can extend $y'(x)$ continuously and
write ${y'(0)=0}$. By Rolle's Theorem, there is an ${x_1 \in (0,a)}$
such that ${y''(x_1)=0}$ and hence ${h(x_1)=0}$. To show the
uniqueness of $x_1$, suppose that ${x_2\in(0,a)}$ is such that
${h(x_2)=0}$. Now, since ${H(x_2)<y(x_2)<\alpha(x_2)}$, by virtue of
the isocline structure ${0<y'(x_2)<\vep^{-1}}$. Moreover, we see
from \eqref{eq11.014} that ${h'(x_2)<0}$. By Lemma~\ref{lem4.001},
we can conclude ${x_2=x_1}$. Finally, by continuity we can conclude
that ${h(x)>0}$ on $(0,x_1)$ and ${h(x)<0}$ on $(x_1,a]$ since
${h(a) = y(a) \sqb{ \vep y(a) - 2 a } < 0}$.
\end{proof}

\begin{proposition}
Let $y$ be a solution to \eqref{eq11.002} strictly between $\cM$ and
$V$ with domain $(0,a)$, where ${a>0}$ and ${y(a-)=V(a)}$. Then, $y$
is concave up on $(0,a)$.
\end{proposition}

\begin{proof}
Let $h$ be as in \eqref{eq11.013} defined with respect to the
solution $y$. Define the sets
\[
    I := \setb{ ~ x \in (0,a) \st \cM(x) < y(x) < \alpha(x) ~ }
\]
and
\[
    J := \setb{ ~ x \in (0,a) \st \alpha(x) \leq y(x) < V(x) ~ }.
\]
Since since ${y(a-)=V(a)}$, we know that ${J\ne\emptyset}$. We will
show separately that ${h(x)>0}$ for each ${x \in I}$ and for each
${x \in J}$.

Suppose that ${x \in J}$. By virtue of the isocline structure,
${y'(x)\geq\vep^{-1}}$. However, any isocline $w$ satisfies
${w'(x)<\vep^{-1}}$.  Appealing to Claim~\ref{claim0006},
${h(x)>0}$.

If ${I=\emptyset}$ then we are done. Suppose that ${I\ne\emptyset}$
and let ${x_3 \in J}$ be fixed. We have already shown that
${h(x_3)>0}$. We need to show that ${h(x)>0}$ for all ${x \in I}$.
By continuity, it suffices to show that $h$ has no zeros in $I$.
Suppose, on the contrary, that this is not the case. By
Lemma~\ref{lem4.001}, $h$ has a right-most zero ${x_1 \in I \subset
(0,x_3]}$. Observe that ${y'(x_1)<\vep^{-1}}$, which follows from
the fact that ${y(x_1)<\alpha(x_1)}$. Since ${h(x_1)=0}$ and
${y(x_1)>0}$, \eqref{eq11.014} informs us ${h'(x_1)<0}$.
Consequently, there is an ${x_2\in(x_1,x_3)}$ such that
${h(x_2)<0}$. Since ${h(x_3)>0}$, the Intermediate Value Theorem
implies that $h$ has a zero in $(x_2,x_3)$, which is a contradiction
since $x_1$ is the right-most zero.
\end{proof}

\begin{proposition}
Let $y$ be a solution to \eqref{eq11.002} lying above $V$ with
domain $(0,a)$, where ${a>0}$, ${y(0+)=\infty}$, and ${y(a-)=V(a)}$.
Then, there is a unique ${x_1\in(0,a)}$ such that ${y''(x_1)=0}$.
Moreover, $y$ is concave up on $(0,x_1)$ and concave down on
$(x_1,a)$.
\end{proposition}

\begin{proof}
We know that
\[
    \lim_{x \to 0^+} y'(x) = -\infty
    \mymbox{and}
    \lim_{x \to a^-} y'(x) = -\infty.
\]
By continuity, there are ${b_1,b_2\in\Rd{}}$ such that
${0<b_1<b_2<a}$ and ${y'(b_1)=y'(b_2)}$. Thus by Rolle's Theorem,
there is an ${x_1\in(0,a)}$ such that ${y''(x_1)=0}$.

To prove uniqueness, suppose that ${x_2\in(0,a)}$ is such that
${h(x_2)=0}$, where $h$ is as in \eqref{eq11.013} defined with
respect to the solution $y$. Now, we know from the isocline
structure that ${y'(x_2)<\vep^{-1}}$. Using \eqref{eq11.014},
${h'(x_2)<0}$. Since $x_2$ was an arbitrary zero of $h$, we can
conclude using Lemma~\ref{lem4.001} that ${x_2=x_1}$. Furthermore,
since ${h'(x_1)<0}$, we can say that $y$ is concave up on $(0,x_1)$
and concave down on $(x_1,a)$.
\end{proof}

\subsection{Curve of Inflection Points}

We know from Table~\ref{tab11.001} that solutions to the scalar
differential equation \eqref{eq11.002} can only have inflection
points between $H$ and $\cM$ or above $V$. We can construct a curve
of inflection points, between $H$ and $\cM$, which is close to the
slow manifold.

It is easily verified that
\[
    h(x,y)
    = \frac{ \vep^2 y^3 - \rb{ 3 \vep x } y^2 + \rb{ 2 x^2 - \vep x^2 - x } y + x^3 }{ \vep y - x },
\]
where $h$ is as in \eqref{eq11.013}. Thus, there are three curves
along which solutions have zero second derivative, given implicitly
by
\[
    \vep^2 y^3 - \rb{ 3 \vep x } y^2 + \rb{ 2 x^2 - \vep x^2 - x } y + x^3
    = 0.
\]
One curve lies below the $x$-axis and is discarded. The other two
curves, as expected, are in the positive quadrant. See
Figure~\ref{fig11.004}.

\begin{figure}[t]
\begin{center}
    \includegraphics[width=\smallfigwid]{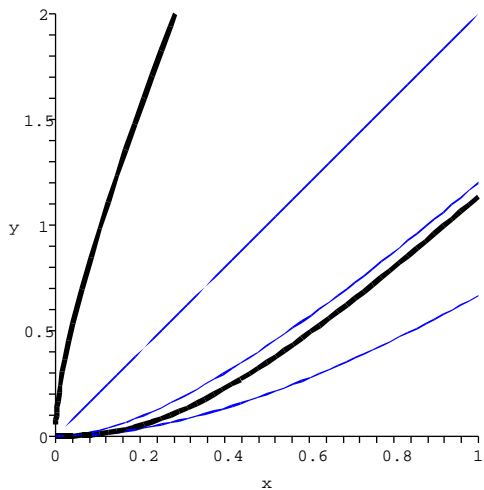}
    \caption{The two thick curves are curves along which solutions of \eqref{eq11.002}
    have inflection points, for parameter value ${\vep=0.5}$.  The thin curves
    are the horizontal, $\alpha$, and vertical isoclines. We will show
    later that the slow manifold lies between the lower
    thick curve and the middle thin curve.} \label{fig11.004}
\end{center}
\end{figure}

Recall that, for a fixed ${c\in\rb{0,\vep^{-1}}}$, the isocline
${w(x):=F(x,c)}$ switches from being a strong lower fence to being a
strong upper fence at ${x=\xi(c)}$ and ${y=F(\xi(c),c)}$, where $F$
is defined in \eqref{eq11.006} and $\xi$ is defined in
\eqref{eq11.010}. As it turns out,
\begin{equation} \label{eq11.015}
    \cY(x) := F\rb{ x, \xi^{-1}(x) },
    \quad
    x > 0
\end{equation}
will be a curve of inflection points between $H$ and $\cM$. Note
that ${H(x)<F(x,c)<\alpha(x)}$ for all ${x>0}$ and
${c\in\rb{0,\vep^{-1}}}$, which follows from the isocline structure.
Moreover, note ${0<\xi^{-1}(x)<\vep^{-1}}$ for all ${x>0}$. Thus,
${H(x)<\cY(x)<\alpha(x)}$ for all ${x>0}$.

\begin{claim} \label{claim0007}
Suppose that ${x_0>0}$ and ${H(x_0)<y_0<\alpha(x_0)}$. Define the
slope ${c:=f(x_0,y_0)}$ and isocline ${w(x):=F(x,c)}$. Then, the
isocline $w$ satisfies
\[
    w'(x_0)
    \begin{cases}
        > f(x_0,y_0), \mymbox{if} H(x_0) < y_0 < \cY(x_0) \\
        = f(x_0,y_0), \mymbox{if} y_0 = \cY(x_0) \\
        < f(x_0,y_0), \mymbox{if} \cY(x_0) < y_0 < \alpha(x_0)
    \end{cases}.
\]
\end{claim}

\begin{proof}
Note that ${0<c<\vep^{-1}}$ and ${y_0=w(x_0)}$. We will only show
the third case since the other two cases are similar. Assume that
${\cY(x_0)<y_0<\alpha(x_0)}$. Appealing to the isocline structure,
we know ${f_y(x,y)>0}$ if ${x>0}$ and ${H(x)<y<\alpha(x)}$.
Consequently, ${f(x_0,y_0)>f(x_0,\cY(x_0))}$. Since ${c=f(x_0,y_0)}$
and ${\xi^{-1}(x_0)=f(x_0,\cY(x_0))}$, we can conclude
${c>\xi^{-1}(x_0)}$. Since $\xi$ is strictly increasing,
${x_0<\xi(c)}$. By virtue of Proposition~\ref{prop0016}, we can
conclude ${w'(x_0)<f(x_0,y_0)}$.
\end{proof}

\begin{claim}
The curve ${y=\cY(x)}$ is analytic for all ${x>0}$.
\end{claim}

\begin{proof}
We know that $\xi^{-1}(x)$ is analytic and
${0<\xi^{-1}(x)<\vep^{-1}}$ for all ${x>0}$. Since $F(x,c)$ is
analytic if ${x>0}$ and ${0<c<\vep^{-1}}$, we see from the
definition \eqref{eq11.015} that $\cY(x)$ is analytic for all
${x>0}$.
\end{proof}

\begin{proposition} \label{prop0017}
The function $h$, defined in \eqref{eq11.013}, satisfies
\[
    h(x,y)
    \begin{cases}
        < 0, \mymbox{if} x > 0, ~ H(x) < y < \cY(x) \\
        = 0, \mymbox{if} x > 0, ~ y = \cY(x) \\
        > 0, \mymbox{if} x > 0, ~ \cY(x) < y < \alpha(x)
    \end{cases}.
\]
\end{proposition}

\begin{proof}
Let ${x_0>0}$ and ${H(x_0)<y_0<\alpha(x_0)}$ be fixed. Consider the
slope ${c:=f(x_0,y_0)}$ and isocline ${w(x):=F(x,c)}$. We know from
Claim~\ref{claim0006} that
\[
    h(x_0,y_0) = x_0 \sqb{ f(x_0,y_0) - w'(x_0) }.
\]
The result follows from Claim~\ref{claim0007}.
\end{proof}

\begin{proposition} \label{prop0018}
The curve ${y=\cY(x)}$ satisfies
\[
    H(x) < \cY(x) < \cM(x)
    \mymbox{for all}
    x > 0.
\]
\end{proposition}

\begin{proof}
We know already that ${H(x)<\cY(x)<\alpha(x)}$ for all ${x>0}$. We
know from our results on concavity (see Table~\ref{tab11.001}) that
${h(x,y)>0}$ if ${x>0}$ and ${\cM(x)<y<\alpha(x)}$, where $h$ is the
function defined in \eqref{eq11.013}. By continuity, we can conclude
${h(x,\cM(x)) \geq 0}$ for all ${x>0}$. It follows from
Proposition~\ref{prop0017} that ${\cY(x)\leq\cM(x)}$ for all
${x>0}$.

To establish a strict inequality, let $h$ be defined along the
solution ${y=\cM(x)}$. Assume, on the contrary, that there is an
${x_0>0}$ such that ${h(x_0)=0}$. Using \eqref{eq11.014},
${h'(x_0)<0}$. This contradicts the fact that ${h(x) \geq 0}$ for
all ${x>0}$.
\end{proof}

\subsection{Slow Tangent Manifold}

The curve ${y=\cY(x)}$ can be referred to as a slow tangent manifold
(or as an intrinsic low-dimensional manifold) since it consists of
the points for which the tangent vector for the planar system
\eqref{eq11.001} points in the slow direction. See, for example,
\cite{Kaper2002,MaasPope,Okuda1,Okuda2,Okuda3,Roussel1994}. To see
why the curve of inflection points and the slow tangent manifold are
equivalent, first consider the general planar system
${\dot{\vec{x}}=\vec{g}(\vec{x})}$, where
${\dot{}=\derivslash{}{t}}$ and ${\vec{g}\in\Cr{1}{\Rd{2},\Rd{2}}}$,
along with the corresponding scalar system
${y'=\fracslash{g_2(x,y)}{g_1(x,y)}}$, where ${'=\derivslash{}{x}}$.
Consider the linearization matrix
${\mat{A}(\vec{x}):=\setb{g_{ij}(\vec{x})}_{i,j=1}^2}$, where
${g_{ij}(\vec{x}):=\pderivslash{g_i(\vec{x})}{x_j}}$, which has
characteristic equation
\[
    \lambda^2 - \tau \lambda + \Delta = 0,
    \mymbox{where}
    \tau := g_{11} + g_{22}
    \mymbox{and}
    \Delta := g_{11} g_{22} - g_{12} g_{21}.
\]
For notational brevity, we are suppressing the dependence on
$\vec{x}$.

\begin{claim}
Suppose ${4\Delta<\tau^2}$ and ${g_{12} \ne 0}$. Then, $\mat{A}$ has
real distinct eigenvalues and associated distinct eigenvectors
given, respectively, by
\[
    \lambda_\pm := \frac{ \tau \pm \sqrt{ \tau^2 - 4 \Delta } }{2}
    \mymbox{and}
    \vec{v}_\pm := \twobyonematrix{1}{\sigma_\pm},
    \mymbox{where}
    \sigma_\pm := \frac{ \lambda_\pm - g_{11} }{ g_{12} }.
\]
\end{claim}

\begin{proof}
The proof is routine.
\end{proof}

\begin{proposition} \label{prop0011}
Suppose that ${g_1 \ne 0}$, ${g_{12} \ne 0}$, and ${4\Delta<\tau^2}$
at some fixed point $(a,b)$ and let $y(x)$ be the scalar solution
through $(a,b)$. Then, ${y''(a)=0}$ if and only if
${\vec{g}~\|~\vec{v}_+}$ or ${\vec{g}~\|~\vec{v}_-}$ at $(a,b)$.
\end{proposition}

\begin{proof}
First, note that $\sigma_\pm$ is the slope of the eigenvector
$\vec{v}_\pm$. If we differentiate the scalar differential equation
and manipulate the resulting expression, we obtain
\[
    y''
    = \frac{ -g_{12} \sqb{ \rb{\fracslash{g_2}{g_1}} - \sigma_+ } \sqb{ \rb{\fracslash{g_2}{g_1}} - \sigma_- } }{ g_1 }.
\]
The conclusion follows.
\end{proof}

For the specific planar and scalar systems \eqref{eq11.001} and
\eqref{eq11.002}, we have
\[
    \mat{A} = \twobytwomatrix{ -2 x + \vep y }{ \vep x }{ 2 x - \vep y }{ -\vep x - 1 },
    \quad
    \tau = -\rb{ \vep + 2 } x + \vep y - 1,
    \mymbox{and}
    \Delta = 2 x - \vep y.
\]
To apply Proposition~\ref{prop0011}, we need to verify that ${g_1
\ne 0}$, ${\pderivslash{g_1}{y} \ne 0}$, and ${\tau^2>4\Delta}$ in
the relevant regions. Trivially, ${g_1 \ne 0}$ (except along the
vertical isocline) and ${g_{12}>0}$ for ${x>0}$. To show that
${\tau^2>4\Delta}$ for ${x>0}$, observe
\[
    \tau^2 - 4 \Delta
    = \vep^2 \sqb{ y - \frac{ \rb{ \vep + 2 } x - 1 }{\vep} }^2 + 4 \vep x
    \geq 4 \vep x
    > 0.
\]
This establishes that $\cY$ is a tangent manifold. To establish that
$\cY$ is indeed a slow tangent manifold, we note that (as can be
shown) ${\lambda_-<\lambda_+<0}$, ${\sigma_-<0<\sigma_+}$, and
${\fracslash{g_2}{g_1}>0}$ for every ${\vec{x}\in\Gamma_1}$. We have
thus demonstrated the following.

\begin{proposition}
Consider the planar system \eqref{eq11.001} and the scalar system
\eqref{eq11.002}. The curve of inflection points ${y=\cY(x)}$ is a
slow tangent manifold.
\end{proposition}

\section{Behaviour of Solutions Near the Origin} \label{sec006}

In this section, we establish the full asymptotic behaviour of
scalar solutions $y(x)$ as ${x \to 0^+}$. Moreover, we will obtain
the leading-order behaviour of planar solutions $\vec{x}(t)$ as
${t\to\infty}$.

\subsection{Scalar Solutions}

We will begin by attempting to find a Taylor series solution.
Consider the differential equation \eqref{eq11.002}, which can be
rewritten
\begin{equation} \label{eq11.016}
    \vep x y y' - x^2 y' - x^2 + y + \vep x y = 0.
\end{equation}
Assume that $y(x)$ is a solution in $\Gamma_0$ of the form
\begin{equation} \label{eq11.017}
    y(x) = \sum_{n=0}^\infty b_n x^n
\end{equation}
for undetermined coefficients $\seq{b}{n}{0}{\infty}$. If we
substitute the series \eqref{eq11.017} into \eqref{eq11.016} and
then solve for the coefficients, we obtain
\begin{align}
    &b_0 = 0, \quad b_1 = 0, \quad b_2 = 1, \quad b_3 = 2 - \vep, \notag\\
    \mymbox{and}
    &b_n = \rb{ n - 1 - \vep } b_{n-1} - \vep \sum_{m=2}^{n-2} \rb{n-m} b_m b_{n-m} \mymbox{for} n \geq 4. \label{eq11.019}
\end{align}

We will use centre manifold theory to show that the series
\eqref{eq11.017} is fully correct for each solution inside the
trapping region $\Gamma_0$. However, we must first show that each
solution is a centre manifold. That is, we must show that each
solution $y(x)$ satisfies ${y(0)=0}$ and ${y'(0)=0}$.
Proposition~\ref{c10t-002} already established that this is true for
$y(x)$ inside $\Gamma_1$.

\begin{proposition} \label{prop0020}
Let $y$ be a solution to \eqref{eq11.002} lying inside $\Gamma_0$
for ${x\in(0,a)}$, where ${a>0}$. Then, we can extend $y(x)$ and
$y'(x)$ to say ${y(0)=0}$ and ${y'(0)=0}$.
\end{proposition}

\begin{proof}
These limits have already been established if $y$ is the slow
manifold or if $y$ lies below the slow manifold $\cM$.  Hence, we
will assume that
\[
    \cM(x) < y(x) < V(x)
    \mymbox{for all}
    x \in (0,a).
\]
Let ${c:=y'\rb{\fracslash{a}{2}}}$. We know from
Table~\ref{tab11.001} that $y$ is concave up on
$\rb{0,\fracslash{a}{2}}$. Since ${\cM(x)>0}$ for
${x\in\rb{0,\fracslash{a}{2}}}$, we thus have
\[
    0 < y(x) < F(x,c)
    \mymbox{for all}
    x \in \rb{ 0, \shalf \, a },
\]
where $F$ is the function given in \eqref{eq11.006}. Note that
\[
    \lim_{x \to 0^+} F(x,c) = 0
    \mymbox{and}
    \lim_{x \to 0^+} \frac{F(x,c)}{x} = 0.
\]
It follows from the Squeeze Theorem that we can take ${y(0)=0}$.
Now, observe that
\[
    0 < \frac{y(x)-y(0)}{x-0} < \frac{F(x,c)}{x}.
\]
Again by the Squeeze Theorem, we see that we can take ${y'(0)=0}$.
\end{proof}

\begin{theorem} \label{thm0012}
Let $y(x)$ be a scalar solution to \eqref{eq11.002} lying inside
$\Gamma_0$ and consider the coefficients $\seq{b}{n}{2}{\infty}$
given in \eqref{eq11.019}. Then,
\[
    y(x) \sim \sum_{n=2}^\infty b_n x^n
    \mymbox{as}
    x \to 0^+.
\]
\end{theorem}

\begin{proof}
The Centre Manifold Theorem guarantees that there is a solution
$u(x)$ to \eqref{eq11.002} such that
\[
    u(x) \sim \sum_{n=2}^\infty b_n x^n
    \mymbox{as}
    x \to 0^+.
\]
Note that the Taylor coefficients of the series for $u(x)$ must be
$\seq{b}{n}{2}{\infty}$ since they are generated uniquely by the
differential equation. Since $y(x)$ is a centre manifold, it follows
from centre manifold theory that
\[
    y(x) - u(x) = \bigO{x^k}
    \mymbox{as}
    x \to 0^+
\]
for any ${k\in\setb{2,3,\ldots}}$. See, for example, Theorem~1 on
page~16, Theorem~3 on page~25, and properties (1) and (2) on page~28
of \cite{Carr}. The conclusion of the theorem follows.
\end{proof}

\begin{remark}
For analytic systems of ordinary differential equations for which
the Centre Manifold Theorem applies, if the Taylor series for a
centre manifold has a nonzero radius of convergence, then the centre
manifold is unique.  Since all solutions  $y$ to \eqref{eq11.002}
lying inside $\Gamma_0$ are centre manifolds, we can conclude that
the Taylor series  ${\sum_{n=2}^\infty b_n x^n}$ has radius of
convergence zero.
\end{remark}

\subsection{Planar Solutions}

We can use the isoclines to extract the leading-order behaviour of
planar solutions as time tends to infinity.

\begin{proposition} \label{prop0005}
Let $\vec{x}(t)$ be the planar solution to \eqref{eq11.001} with
initial condition ${\vec{x}(0)=\vec{x}_0}$, where
${\vec{x}_0\in\setb{\vec{x} \in S \st x > 0}}$. Then,
\[
    x(t) = \frac{1}{t} + \vep \, \frac{\lnb{t}}{t^2} + \litO{\frac{\lnb{t}}{t^2}}
    \mymbox{and}
    y(t) = \frac{1}{t^2} + 2 \vep \, \frac{\lnb{t}}{t^3} + \litO{\frac{\lnb{t}}{t^3}}
    \mymbox{as}
    t \to \infty.
\]
\end{proposition}

\begin{proof}
Let ${c>0}$ be fixed and arbitrary. We know from
Theorem~\ref{thm0011}, Table~\ref{tab11.001}, and the isocline
structure that there exists a ${T \geq 0}$ such that
\begin{equation} \label{eq12.008}
    H(x(t)) \leq y(t) \leq F(x(t),c)
    \mymbox{for all}
    t \geq T,
\end{equation}
where $F$ is given in \eqref{eq11.006}. Using \eqref{eq11.001},
\eqref{eq11.004}, \eqref{eq11.006}, and \eqref{eq12.008}, we can see
that $x(t)$ satisfies
\begin{equation} \label{eq-002}
    -\frac{ x(t)^2 }{ 1 + \vep x(t) } \leq \dot{x}(t) \leq -\frac{ x(t)^2 }{ 1 + b x(t) }
    \mymbox{for all}
    t \geq T,
\end{equation}
where ${b:=\fracslash{\vep}{K(c)}}$ and $K$ is the function defined
in \eqref{eq11.005}. Note that the solution of the initial value
problem
\[
    \dot{u} = -\frac{u^2}{1+au},
    \quad
    u(t_0) = u_0,
\]
where ${a,u_0>0}$ and ${t_0 \geq 0}$ are constants, is
\[
    u(t) = \varphi(t;a,t_0,u_0) := \frac{ 1 }{ a \, \LambertW{ \sqb{ \frac{1}{au_0} \, \expb{\frac{1}{au_0}} } \expb{ \frac{t-t_0}{a} } } },
\]
where $W$ is the Lambert $W$ function \cite{Corless}. A simple
comparison argument applied to \eqref{eq-002} establishes
\begin{equation} \label{eq12.009}
    \varphi(t;\vep,T,x(T)) \leq x(t) \leq \varphi(t;b,T,x(T))
    \mymbox{for all}
    t \geq T.
\end{equation}

A standard property of the Lambert $W$ function is
\[
    \LambertW{t} = \lnb{t} - \lnb{\lnb{t}} + \litO{\lnb{\lnb{t}}}
    \mymbox{as}
    t \to \infty.
\]
Consequently, it can be shown
\[
    \LambertW{\exp{t}} = t - \lnb{t} + \litO{\lnb{t}}
    \mymbox{and}
    \frac{1}{\LambertW{\exp{t}}} = \frac{1}{t} + \frac{\lnb{t}}{t^2} + \litO{\frac{\lnb{t}}{t^2}}
    \mymbox{as}
    t \to \infty.
\]
With a little manipulation, it can be verified that
\[
    \varphi(t;a,t_0,u_0) = \frac{1}{t} + a \, \frac{\lnb{t}}{t^2} + \litO{\frac{\lnb{t}}{t^2}}
    \mymbox{as}
    t \to \infty
\]
for any ${a,u_0>0}$ and ${t_0 \geq 0}$. It follows from
\eqref{eq12.009} that
\[
    \liminf_{t\to\infty} \sqb{ x(t) - \frac{1}{t} } \sqb{ \frac{t^2}{\lnb{t}} } \geq \vep
    \mymbox{and}
    \limsup_{t\to\infty} \sqb{ x(t) - \frac{1}{t} } \sqb{ \frac{t^2}{\lnb{t}} } \leq b.
\]
Since $c$ was arbitrary with ${K(c) \to 1}$ and ${b\to\vep}$ as ${c
\to 0^+}$,
\[
    \lim_{t\to\infty} \sqb{ x(t) - \frac{1}{t} } \sqb{ \frac{t^2}{\lnb{t}} }
    = \vep.
\]
This yields the desired conclusion for $x(t)$. The conclusion for
$y(t)$ follows from Theorem~\ref{thm0012}.
\end{proof}

It is possible to derive the expression for $x(t)$ in
Proposition~\ref{prop0005} without appealing to the isocline
structure and concavity. To achieve this, we will note that $x(t)$
satisfies the integral equation
\begin{equation} \label{eq12.005}
    \frac{1}{x(t)} - \frac{1}{x_0} = t - \vep \integral{0}{t}{\frac{y(s)}{x(s)}}{s}
\end{equation}
and then twice utilize the following easy-to-verify lemma.

\begin{lemma} \label{lem0001}
Let ${a\in\Rd{}}$ be a constant and let ${f,g:[a,\infty)\to\Rd{}}$
be nonnegative, integrable functions such that
${f(t)=g(t)+\litO{g(t)}}$ as ${t\to\infty}$. If $G(t)$ is an
antiderivative of $g(t)$ such that ${G(t)\to\infty}$ as
${t\to\infty}$, then ${\integral{a}{t}{f(s)}{s}=G(t)+\litO{G(t)}}$
as ${t\to\infty}$.
\end{lemma}

We know from Theorems~\ref{thm0011} and \ref{thm0012} that
\[
    x(t) = \litO{1},
    \quad
    y(t) = \litO{1},
    \mymbox{and}
    \frac{y(t)}{x(t)} + 1 = 1 + \litO{1}
    \mymbox{as}
    t \to \infty.
\]
It follows from Lemma~\ref{lem0001} that
\[
    \integral{0}{t}{\sqb{\frac{y(s)}{x(s)}+1}}{s}
    = t + \litO{t}
    \implies
    \integral{0}{t}{\frac{y(s)}{x(s)}}{s}
    = \litO{t}
    \mymbox{as}
    t \to \infty.
\]
By virtue of the integral equation \eqref{eq12.005},
\[
    \frac{1}{x(t)} = t \sqb{ 1 + \litO{1} }
    \implies
    x(t) = \frac{1}{t} + \litO{\frac{1}{t}}
    \mymbox{as}
    t \to \infty.
\]
To take this one step further, observe now that
\[
    \frac{y(t)}{x(t)} = \frac{1}{t} + \litO{\frac{1}{t}}
    \mymbox{as}
    t \to \infty,
\]
which follows from Theorem~\ref{thm0012}, and so by
Lemma~\ref{lem0001} we have
\[
    \integral{0}{t}{\frac{y(s)}{x(s)}}{s}
    = \lnb{t} + \litO{\lnb{t}}
    \mymbox{as}
    t \to \infty.
\]
By virtue of the integral equation \eqref{eq12.005} once again,
\[
    \frac{1}{x(t)} = t \sqb{ 1 - \vep \, \frac{\lnb{t}}{t} + \litO{\frac{\lnb{t}}{t}} }
    \implies
    x(t) = \frac{1}{t} + \vep \, \frac{\lnb{t}}{t^2} + \litO{\frac{\lnb{t}}{t^2}}
    \mymbox{as} t \to \infty.
\]

\section{All Solutions Must Enter the Antifunnel} \label{sec007}

Earlier, in Theorem~\ref{thm0011}, we showed that all solutions
$\vec{x}(t)$ to the planar system \eqref{eq11.001}, except for the
trivial solutions, eventually enter the trapping region $\Gamma_0$.
Here, we show that $\Gamma_1$ is itself a trapping region.

\begin{samepage}
\begin{theorem}
Let $\vec{x}(t)$ be the solution to \eqref{eq11.001} with
${\vec{x}(0)=\vec{x}_0}$, where ${\vec{x}_0\in\setb{\vec{x} \in S
\st x>0}}$.
\begin{enumerate}[(a)]
    \item
        There is a ${t^* \geq 0}$ such that ${\vec{x}(t)\in\Gamma_1}$ for all ${t
        \geq t^*}$.
    \item
        Define the region
        \[
            \Gamma_2 := \setb{ ~ (x,y) \st x > 0, ~ \cY(x) \leq y \leq \alpha(x) ~ }.
        \]
        Then, there is a ${t^* \geq 0}$ such that ${\vec{x}(t)\in\Gamma_2}$ for all ${t
        \geq t^*}$.
\end{enumerate}
\end{theorem}
\end{samepage}

\needspace{2.0cm}
\begin{proof}
\begin{enumerate}[(a)]
    \item
        We know from Theorem~\ref{thm0011} that $\vec{x}(t)$ eventually
        enters and stays in $\Gamma_0$. Let $y(x)$ be the corresponding
        scalar solution to \eqref{eq11.002}. Then, we can say ${y'(0)=0}$.
        Appealing to the isocline structure, this means that $\vec{x}(t)$
        has entered $\Gamma_1$. Furthermore, since
        ${\vec{g}\dotprod\normal<0}$ along the horizontal and $\alpha$
        isoclines which form the boundaries of the region in question, we
        see that $\Gamma_1$ is positively invariant.
    \item
        It follows from Table~\ref{tab11.001},
        Proposition~\ref{prop0017}, and the previous part of the theorem.
\end{enumerate}
\end{proof}

\section{Properties of the Slow Manifold} \label{sec008}

In this section, we will highlight some properties of the slow
manifold.

\begin{proposition} \label{prop0019}
The slow manifold ${y=\cM(x)}$ satisfies, for all ${x>0}$,
\[
    0 < H(x) < \cY(x) < \cM(x) < \alpha(x).
\]
Furthermore,
\[
    \lim_{x \to 0^+} \cM(x) = 0.
\]
\end{proposition}

\begin{proof}
The first part follows from Theorem~\ref{c10t-001} and
Proposition~\ref{prop0018}. The second part follows from the Squeeze
Theorem.
\end{proof}

\begin{proposition}
The slow manifold ${y=\cM(x)}$ is concave up for all ${x>0}$.
\end{proposition}

\begin{proof}
It follows from Propositions~\ref{prop0017} and \ref{prop0019} that
${h(x,\cM(x))>0}$ for all ${x>0}$, where $h$ is the function defined
in \eqref{eq11.013}. Since ${\sgn{\cM''(x)}=\sgn{h(x,\cM(x))}}$, it
must be that the slow manifold is concave up for all ${x>0}$.
\end{proof}

\begin{proposition}
The slope of the slow manifold ${y=\cM(x)}$ satisfies
\[
    0 < \cM'(x) < \vep^{-1}
    \mymbox{for all}
    x > 0.
\]
Furthermore,
\[
    \lim_{x \to 0^+} \cM'(x) = 0
    \mymbox{and}
    \lim_{x\to\infty} \cM'(x) = \vep^{-1}.
\]
\end{proposition}

\begin{proof}
The first part is a consequence of Proposition~\ref{prop0019} and
the isocline structure. The first limit is a special case of
Proposition~\ref{prop0020}. To prove the second limit, let
${c\in\rb{0,\vep^{-1}}}$. It follows from Proposition~\ref{prop0016}
and the isocline structure that
\[
    c < \cM'(x) < \vep^{-1}
    \mymbox{for all}
    x > \xi(c),
\]
where $\xi$ is the function defined in \eqref{eq11.010}. Applying
\eqref{eq11.011} and the Squeeze Theorem gives the second limit.
\end{proof}

\begin{remark}
The justification which Fraser provides in \cite{Fraser1988} (just
before Theorem~1) that ${\cM'(x)\to\vep^{-1}}$ as ${x\to\infty}$ is
incorrect. The error is that the distance between the horizontal and
vertical isoclines does not tend to zero as $x$ tends to infinity.
Thus, the asymptotic behaviour of $\cM'(x)$ need not be the same as
the asymptotic behaviour of $H'(x)$ and $V'(x)$.
\end{remark}

\begin{proposition}
Asymptotically, the slow manifold can be written
\[
    \cM(x) \sim \sum_{n=2}^\infty b_n x^n
    \mymbox{as}
    x \to 0^+,
\]
where the coefficients $\seq{b}{n}{2}{\infty}$ are as in
\eqref{eq11.019}.
\end{proposition}

\begin{proof}
Since the slow manifold is contained entirely in $\Gamma_0$, we can
apply Theorem~\ref{thm0012}.
\end{proof}

\begin{corollary} \label{cor0005}
The slow manifold satisfies
\[
    \cM(x) = H(x) + \bigO{x^3}
    \mymbox{as}
    x \to 0^+.
\]
Moreover, this statement would not be true if we replace $H(x)$ with
any other isocline $F(x,c)$.
\end{corollary}

\begin{proof}
It follows from a comparison of the asymptotic expansions for
$\cM(x)$, $H(x)$, and $F(x,c)$.
\end{proof}

Now we will establish the full asymptotic behaviour of the slow
manifold at infinity. First, we will extract as much information as
possible from the isoclines. Second, we will attempt to find a
series in integer powers of $x$. Third, we will prove definitively
that the resulting series is indeed fully correct.

Let ${c\in\rb{0,\vep^{-1}}}$. We know from
Proposition~\ref{prop0016} that
\[
    F(x,c) < \cM(x) < \alpha(x)
    \mymbox{for all}
    x > \xi(c),
\]
where $F$ is defined in \eqref{eq11.006} and $\xi$ is defined in
\eqref{eq11.010}. Note that
\[
    F(x,c) = \frac{x}{\vep} - \frac{K(c)}{\vep^2} + \bigO{\frac{1}{x}}
    \mymbox{and}
    \alpha(x) = \frac{x}{\vep} - \frac{1}{\vep\rb{1+\vep}} + \bigO{\frac{1}{x}}
    \mymbox{as}
    x \to \infty.
\]
Since
\[
    F(x,c) - \frac{x}{\vep} < \cM(x) - \frac{x}{\vep} < \alpha(x) - \frac{x}{\vep}
    \mymbox{for all}
    x > \xi(c),
\]
we can conclude
\[
    \liminf_{x\to\infty} \sqb{ \cM(x) - \frac{x}{\vep} } \geq -\frac{K(c)}{\vep^2}
    \mymbox{and}
    \limsup_{x\to\infty} \sqb{ \cM(x) - \frac{x}{\vep} } \leq -\frac{1}{\vep\rb{1+\vep}}.
\]
Since ${c\in\rb{0,\vep^{-1}}}$ is arbitrary and
\[
    \lim_{c\to\rb{\vep^{-1}}^-} -\frac{K(c)}{\vep^2} = -\frac{1}{\vep\rb{1+\vep}},
\]
we have
\[
    \lim_{x\to\infty} \sqb{ \cM(x) - \frac{x}{\vep} } = -\frac{1}{\vep\rb{1+\vep}}
    \mymbox{and}
    \cM(x) = \frac{x}{\vep} - \frac{1}{\vep\rb{1+\vep}} + \litO{1}
    \mymbox{as}
    x \to \infty.
\]

Assume that we can write
\begin{equation} \label{eq11.026}
    \cM(x) = \sum_{n=-1}^\infty \rho_n x^{-n}
\end{equation}
for undetermined coefficients $\seq{\rho}{n}{-1}{\infty}$. Of
course, we expect ${\rho_{-\!1}=\vep^{-1}}$ and ${\rho_0=-\vep^{-1}\rb{1+\vep}^{-1}}$. Now, write the differential
equation \eqref{eq11.002} as
\begin{equation} \label{eq11.027}
    \vep x y y' - x^2 y' - x^2 + y + \vep x y = 0.
\end{equation}
If we substitute \eqref{eq11.026} into \eqref{eq11.027} and solve
for the coefficients, we obtain
\begin{align}
    &\rho_{-\!1} = \frac{1}{\vep},
    \quad
    \rho_0 = -\frac{1}{\vep\rb{1+\vep}}, \notag\\
    \mymbox{and}
    &\rho_{n} = -\frac{1}{ 1 + \vep } \sqb{ \rho_{n-1} - \vep \sum_{m=1}^{n} \rb{n-m} \rho_{m-1} \rho_{n-m} }
    \mymbox{for}
    n \geq 1. \label{eq11.030}
\end{align}

\begin{proposition}
Asymptotically, the slow manifold can be written
\[
    \cM(x) \sim \sum_{n=-1}^\infty \rho_n x^{-n}
    \mymbox{as}
    x \to \infty,
\]
where the coefficients $\seq{\rho}{n}{-1}{\infty}$ are as in
\eqref{eq11.030}.
\end{proposition}

\begin{proof}
To prove the result, we will apply the Centre Manifold Theorem to a
fixed point at infinity. Consider the change of variables
\[
    X := x^{-1}
    \mymbox{and}
    Y := y - r(x),
    \mymbox{where}
    r(x) := \rho_{-\!1} x + \rho_0 + \rho_1 x^{-1},
\]
with the coefficients $\rho_{-\!1}$, $\rho_0$, and $\rho_1$ being
given in \eqref{eq11.030}. Differentiate the new variables with
respect to time and use the differential equation \eqref{eq11.001}
to obtain the system
\begin{align*}
    \dot{X} &= -X^2 g_1\rb{ X^{-1}, r\rb{X^{-1}} + Y }, \\
    \quad
    \dot{Y} &= -r'\rb{X^{-1}} g_1\rb{ X^{-1}, r\rb{X^{-1}} + Y } + g_2\rb{ X^{-1}, r\rb{X^{-1}} + Y },
\end{align*}
where $g_1$ and $g_2$ are as in \eqref{eq10.009}.  This system is
not polynomial but there is no harm in considering the system
\begin{align} \label{eq11.031}
    \dot{X} &= -X^3 g_1\rb{ X^{-1}, r\rb{X^{-1}} + Y }, \notag\\
    \quad
    \dot{Y} &= X \sqb{ -r'\rb{X^{-1}} g_1\rb{ X^{-1}, r\rb{X^{-1}} + Y } + g_2\rb{ X^{-1}, r\rb{X^{-1}} + Y } },
\end{align}
which is polynomial. This is because the resulting scalar
differential equation will be the same. The system at hand, while
messy, is in the canonical form for the Centre Manifold Theorem.
Note that the eigenvalues of the matrix for the linear part of this
system are $0$ and $-(1+\vep)$. We know from centre manifold theory
that there is a $\Cr{\infty}{}$ centre manifold ${Y=\cC(X)}$ which,
we claim, must be the slow manifold.

For the scalar differential equation in the original coordinates,
all other solutions except for the slow manifold leave the
antifunnel $\Gamma_1$. To establish that the slow manifold in the
original coordinates is the same as the centre manifold in the new
coordinates, we need only show that ${Y=\cC(X)}$ is the only scalar
solution in the new coordinates which is $\litO{1}$ as ${X \to
0^+}$.

Observe that the $Y$-axis is invariant.  Moreover, the fixed point
${(X,Y)=(0,0)}$ is a saddle node (or a degenerate saddle).  The
physically relevant portion of the phase portrait, namely ${X \geq
0}$, consists of two hyperbolic sectors, one with the positive
$Y$-axis and the centre manifold as boundaries and the other with
the negative $Y$-axis and the centre manifold as boundaries. See
Figure~\ref{fig11.009}. This can be shown using techniques in \S9.21
of \cite{Andronov} (in particular Theorem~65 on page~340).
Therefore, $\cC(X)$ is the only scalar solution in the new
coordinates which is $\litO{1}$ as ${X \to 0^+}$. It follows that
the centre manifold is indeed the slow manifold.

\begin{figure}[t]
\begin{center}
    \includegraphics[width=\mediumfigwid]{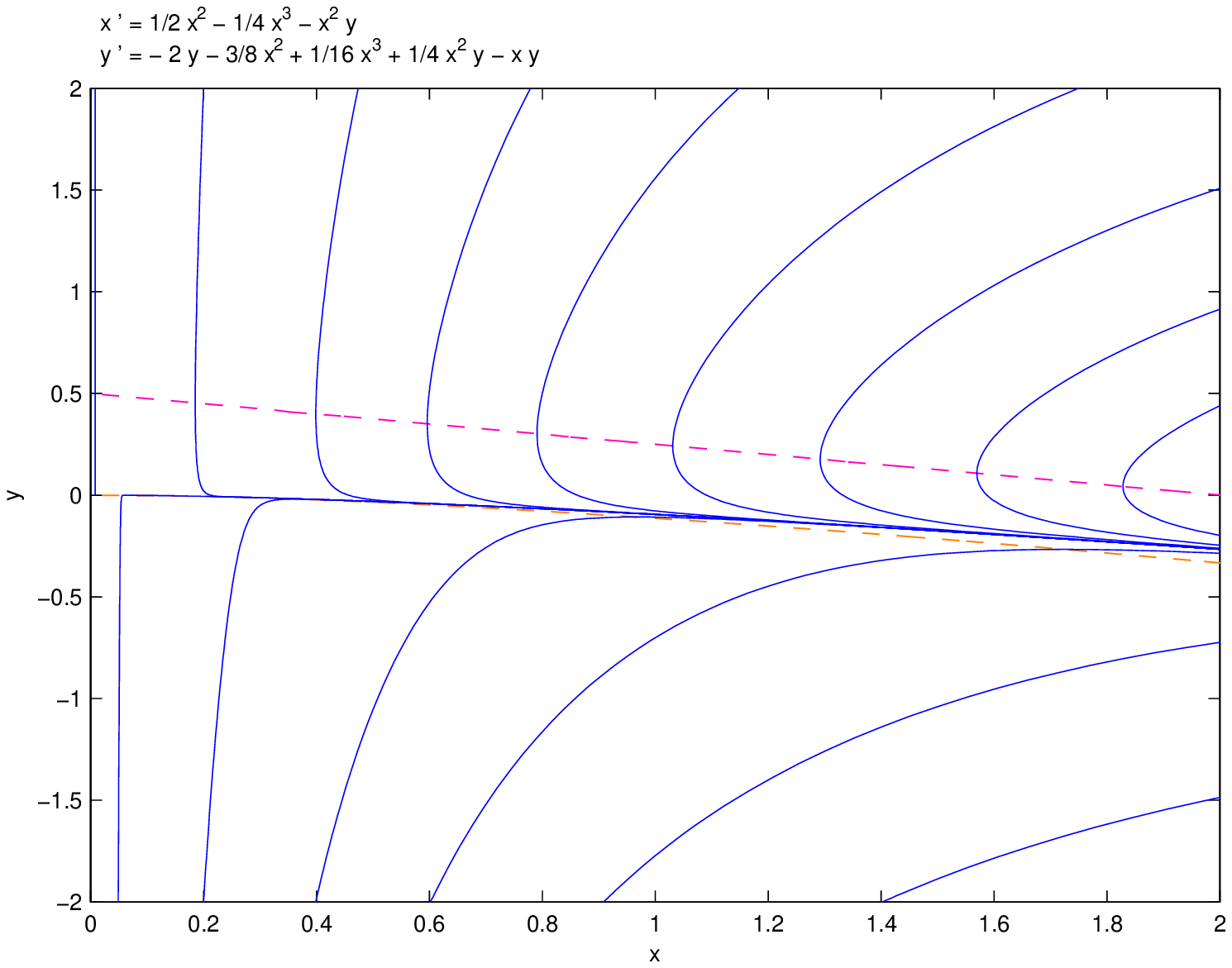}
    \caption[Fixed point at infinity]{A phase portrait for \eqref{eq11.031} for ${\vep=1.0}$.} \label{fig11.009}
\end{center}
\end{figure}

By the Centre Manifold Theorem, in the new coordinates the slow
manifold can be written
\[
    \cM(X) \sim \sum_{n=2}^\infty \hrho_n X^n
    \mymbox{as}
    X \to 0^+,
\]
for some coefficients $\seq{\hrho}{n}{2}{\infty}$.  Upon reverting
back to original coordinates and observing that the coefficients in
\eqref{eq11.030} are generated uniquely from the differential
equation, the conclusion follows.
\end{proof}

\begin{samepage}
\begin{corollary} \label{cor0006}
The slow manifold satisfies
\[
    \cM(x) = \alpha(x) + \bigO{\frac{1}{x^2}}
    \mymbox{as}
    x \to \infty.
\]
Moreover, this statement would not be true if we replace $\alpha(x)$
with any other isocline $F(x,c)$.
\end{corollary}
\end{samepage}

\begin{proof}
It follows from a comparison of the asymptotic expansions for
$\cM(x)$, $\alpha(x)$, and $F(x,c)$.
\end{proof}

\section{Open Questions} \label{sec009}

It would be nice to extend Proposition~\ref{prop0005} to include
more terms. In particular, it is desirable to have the lowest-order
term which depends on the initial condition. For $x(t)$, it is
expected that the initial condition $x_0$ first appears in the
$\fracslash{1}{t^2}$ term since this is the case when ${\vep=0}$,
which has
\[
    x(t) = \frac{x_0}{1+x_0t} \sim \sum_{n=1}^\infty \frac{ \rb{-1}^{n+1} }{ x_0^{n-1} t^n }
    \mymbox{as}
    t \to \infty.
\]
More generally, we would like to develop an iterative procedure to
extract as many terms as possible from the asymptotic expansion of a
solution of a nonlinear differential equation which approaches a
degenerate critical point in the direction of a centre manifold.

\section*{Acknowledgements} \label{sec010}

\addcontentsline{toc}{section}{Acknowledgements}

This paper is primarily based on the majority of Part~III of
\cite{CalderThesis}, which is one of the authors' (Calder) Ph.D.
thesis written under the supervision of the other author (Siegel).
Moreover, this paper parallels the authors' paper
\cite{CalderSiegel01}, which dealt with the Michaelis-Menten
mechanism, in many ways with a number of crucial differences in the
methods and details. Notable topics covered in this paper and not
the former paper include a detailed proof of global asymptotic
stability, nested antifunnels, the construction of the curve of
inflection points, and the leading-order behaviour of planar
solutions as time tends to infinity.

\end{document}